\documentclass[11pt]{amsart}

\usepackage{amssymb}
\usepackage{epsfig}
\usepackage{amsmath}
\usepackage{color}
\usepackage{mathrsfs}
 \usepackage{caption}
\usepackage{delarray}
\usepackage[all]{xy}
\usepackage{mathtools}
\usepackage{graphics}
\usepackage{graphicx}
\usepackage{xkeyval}
\usepackage{color}
\usepackage{keyval}
\usepackage{pict2e}
\usepackage{ifthen}
\usepackage{amsmath,amssymb}
\usepackage{amsmath,amscd}
\usepackage{amstext}
\usepackage{amsxtra}
\usepackage{amsthm}
\usepackage{latexsym}
\usepackage{verbatim}
\usepackage{pstricks}
\usepackage{tikz}
\usetikzlibrary{shapes,snakes}
\usepackage{multicol}
\usepackage{mathrsfs}
\usepackage[all]{xy}
\usepackage[pdftex]{hyperref}
\usepackage{enumerate}
\usepackage{enumitem}
\DeclareMathAlphabet{\mathcalligra}{T1}{calligra}{m}{n}
\numberwithin{equation}{section}
\theoremstyle{plain}
\newtheorem{theorem}{Theorem}[section]

\newtheorem{lemma}[theorem]{Lemma}
\newtheorem{proposition}[theorem]{Proposition}

\theoremstyle{definition}

\newtheorem{definition}[theorem]{Definition}
\newtheorem*{condition-S}{Condition}

\theoremstyle{remark}
\newtheorem{remark}[theorem]{Remark}

\newcommand{\C}{\mathbb{C}}
\newcommand{\EL}{\mathcal{L}}

\newcommand{\ep}{\varepsilon}

\newcommand{\F}{\mathcal{F}}

\objectmargin{1,5ex}

\begin{document}

\title{Moduli Spaces of Germs of Semiquasihomogeneous Legendrian Curves}
\author{Marco Silva Mendes}
\author{Orlando Neto}
\date{2018}

\maketitle

\pagenumbering{arabic}



\begin{abstract}
We construct  a moduli space for Legendrian curves singularities which are contactomorphic-equivalent and equisingular through a contact analogue of the Kodaira-Spencer map for curve singularities. We focus on the specific case of Legendrian curves which are the conormal of a plane curve with one Puiseux pair.

\end{abstract}

\section{Introduction}\label{Intro}
Greuel, Laudal, Pfister et all (see  \cite{KS}, \cite{LaP}) constructed moduli spaces of germs of plane curves equisingular to a plane curve $\{y^k+x^n=0\}$, $(k,n)=1$. Their main tools are the Kodaira Spencer map of the equisingular semiuniversal deformation of the curve and the results of \cite{GP1}. We extend their results to Legendrian curves.

Let $Y$ be the germ of a plane curve that is a generic plane projection of a Legendrian curve $L$. The equisingularity type of $Y$ does not depend on the projection (see \cite{Neto}). Two Legendrian curves are equisingular if their generic plane projections are equisingular. We say that an irreducible Legendrian curve $L$ is semiquasihomogeneous if its generic plane projection is equisingular to a quasihomogeneous plane curve $\{y^k+n^n=0\}$, for some $k,n$ such that $(k,n)=1$. Hence the generic plane projection of $L$ is a semiquasihomogeneous plane curve.

In section \ref{RCG} we recall the main results of relative contact geometry. In section \ref{MKSM} we construct the microlocal Kodaira Spencer map and study its kernel $\EL_B$, a Lie algebra of vector fields over the base space $\C^B$ of the semiuniversal equisingular deformation of the plane curve $\{y^k+n^n=0\}$. We use $\EL_B$ in order to construct a Lie algebra of vector fields $\EL_C$ over the base space $\C^C$ of the microlocal semiuniversal equisingular deformation of $\{y^k+n^n=0\}$. In section \ref{GQ} we recall some results of \cite{GP1}. In section \ref{FS} we study the stratification of $\C^C$ induced by $\EL_C$ and show that the conormals of two fibers $\F_b$, $\F_c$ of the microlocal semiuniversal equisingular deformation of $\{y^k+n^n=0\}$ are isomorphic if and only if $b$ and $c$ are in the same integral manifold of $\EL_C$. Moreover, we construct the moduli spaces. The final section in dedicated to presenting an example.



\section{Relative contact geometry}\label{RCG}

Let $q:X\to S$ be a morphism of complex spaces.
We can associate to $q$ a coherent $\mathcal O_X$-module 
$\Omega^1_{X/S}$, the \emph{sheaf of relative differential forms of $X\to S$},  and a differential morphism $d:\mathcal O_X\to \Omega^1_{X/S}$ (see \cite{HA} or \cite{EDLC}).

If $\Omega^1_{X/S}$ is a locally free $\mathcal O_X$-module, we denote by $\pi=\pi_{X/S}:T^*(X/S)\to X$ the vector bundle with sheaf of sections $\Omega^1_{X/S}$. We say that $T(X/S)$ [$T^*(X/S)$] is the \em relative tangent bundle \em [\em cotangent bundle\/\em ] of $X\to S$. 

Let $\varphi:X_1\to X_2$, $q_i:X_i\to S$ be morphisms of complex spaces such that $q_2\varphi=q_1$. There is a morphism of $\mathcal O_{X_1}$-modules
\begin{equation}\label{ROPHI1-3}
\widehat\rho_\varphi:\varphi^*\Omega^1_{X_2/S}
=\mathcal O_{X_1}\otimes _{\varphi^{-1}\mathcal O_{X_2}}
\varphi^{-1}\Omega^1_{X_2/S} \to
\Omega^1_{X_1/S}.
\end{equation}
If $\Omega^1_{X_i/S}$, $i=1,2$, and the kernel and cokernel of (\ref{ROPHI1-3}) are locally free, we have a morphism of vector bundles
\begin{equation}\label{ROPHI2-3}
\rho_\varphi: X_1\times_{X_2}T^*(X_2/S) \to T^*(X_1/S).
\end{equation}

If $\varphi$ is an inclusion map, we say that the kernel of (\ref{ROPHI2-3}), and its projectivization, are the \em conormal bundle of $X_1$ relative to $S$. \em
We will denote by $T^*_{X_1}(X_2/S)$ or $\mathbb P^*_{X_1}(X_2/S)$ the conormal bundle of $X_1$ relative to $S$.

Assume $M$ is a manifold.
When $q$ is the projection $M\times S\to S$ we will replace "$M\times S/S$" by $"M|S"$. Let $r$ be the projection $M\times S \to M$.
Notice that $\Omega^1_{M|S}\xrightarrow{\sim}\mathcal O_{M\times S}
\otimes_{r^{-1}\mathcal O_M}r^{-1}\Omega^1_M$
is a locally free $\mathcal O_{M\times S}$-module.
Moreover, $T^*(M|S)=T^*M\times S$.

We say that $\Omega^1_{M|S}$ is the \em sheaf of relative differential forms of $M$ over $S$. \em
We say that $T^*(M|S)$ is the \em relative cotangent  bundle of $M$ over $S$. \em

Let $N$ be a complex manifold of dimension $2n-1$. 
Let $S$ be a complex space.
We say that a section $\omega$ of $\Omega^1_{N|S}$ is a \em relative contact form of \em $N$ over $S$ if  $\omega\wedge d\omega^{n-1}$ is a local generator of 
$\Omega^{2n-1}_{N|S}$.
Let $\mathfrak C$ be a locally free subsheaf of $\Omega^1_{N|S}$. 
We say that $\mathfrak C$ is a \em structure of relative contact manifold on \em $N$ over $S$ if $\mathfrak C$ is locally generated by a relative contact form of $N$ over $S$.  We say that $(N\times S,\mathfrak C)$ is a \em relative contact manifold over $S$. \em
When $S$ is a point we obtain the usual notion of contact manifold.

Let $(N_1\times S,\mathfrak C_1)$, $(N_2\times S,\mathfrak C_2)$ be relative contact manifolds over $S$. 
Let $\chi$ be a morphism from $N_1\times S$ into $N_2\times S$ such that 
$q_{N_2}\circ \chi =q_{N_1}$. 
We say that $\chi$ is a \em relative contact transformation \em of $(N_1\times S,\mathfrak C_1)$ into $(N_2\times S,\mathfrak C_2)$ if the pull-back by $\chi$ of each local generator of $\mathfrak C_2$ is a local generator of $\mathfrak C_1$.

We say that the projectivization $\pi_{X/S}:\mathbb P^*(X/S)\to X$ of the vector bundle 
$T^*(X/S)$ is the \em projective cotangent bundle \em of $X \to S$.

Let $(x_1,...,x_n)$ be a partial system of local coordinates on an open set $U$ of $X$. Let $(x_1,...,x_n,\xi_1,...,\xi_n)$ be the associated partial system of symplectic coordinates of $T^*(X/S)$ on $V=\pi^{-1}(U)$.
Set $p_{i,j}=\xi_i\xi_j^{-1}$, $i\not=j$,
$$
V_i=\{ (x,\xi)\in ~V: ~\xi_i\not =0  \}, 
\qquad \omega_i=\xi_i^{-1}\theta,
\qquad i=1,...,n.
$$
each $\omega_i$ defines a relative contact form $dx_j-\sum_{i\not=j}p_{i,j}dx_i$ on $\mathbb P^*(X/S)$, endowing  $\mathbb P^*(X/S)$ with a structure of relative contact manifold over $S$.

Let $\omega$ be a germ at $(x,o)$ of a relative contact form of $\mathfrak C$.
A lifting $\widetilde\omega$ of $\omega$ defines a germ 
$\widetilde{\mathfrak C}$ of a relative contact structure of 
$N\times T_oS \to T_oS$. Moreover, $\widetilde{\mathfrak C}$ is a lifting of 
the germ at $o$ of ${\mathfrak C}$.

Let $(N\times S,\mathfrak C)$ be a relative contact manifold over a complex manifold $S$.  
Assume $N$ has dimension $2n-1$ and $S$ has dimension $\ell$.
Let $\mathcal L$ be a reduced analytic set of $N\times S$ of pure dimension $n+\ell-1$.
We say that $\mathcal L$ is a \em relative Legendrian variety \em of $N\times S$ over $S$ if for each section 
$\omega$ of $\mathfrak C$, $\omega$ vanishes on the regular part of $\mathcal L$. When $S$ is a point, we say that $\mathcal L$ is a \em  Legendrian variety \em of $N$.

Let $\mathcal L$ be an analytic set of $N\times S$. Let $(x,o) \in \mathcal L$. Assume $S$ is an irreducible germ of a complex space at $o$.
We say that  $\mathcal L$ is a \em relative Legendrian variety of $N$ over $S$  at \em $(x,o)$  if there is a relative Legendrian variety 
$\widetilde{\mathcal L}$ of $(N,x)$ over $(T_oS,0)$  that is a lifting of 
the germ of $\mathcal L$ at $(x,o)$. Assume $S$ is a germ of a complex space at $o$ with irreducible components $S_i, i\in I$. We say that  $\mathcal L$ is a \em relative Legendrian variety of $N$ over $S$  at \em $(x,o)$  if $S_i \times_S \mathcal L$ is a relative Legendrian variety of $S_i\times_S N$ over $S_i$ at $(x,o)$, for each $i \in I$.

We say that  $\mathcal L$ is a \em relative Legendrian variety 
 of $N\times S$ \em if $\mathcal L$ is a relative Legendrian variety of $N\times S$ at $(x,o)$ for each $(x,o)\in \mathcal L$.

Let $Y$ be a reduced analytic set of $M$.
Let $\mathcal Y$ be a flat deformation of $Y$ over $S$.
Set $X=M\times S\setminus \mathcal Y_{\rm sing}$.
We say that the Zariski closure of $\mathbb P^*_{\mathcal Y_{\rm reg}}(X/S)$ in $\mathbb P^*(M|S)$ is the \em conormal
$\mathbb P^*_{\mathcal Y}(M|S)$ of $\mathcal Y$ over $S$. \em

\begin{theorem}
The conormal of $\mathcal Y$ over $S$ is a relative Legendrian variety of $\mathbb P^*(M|S)$.
If $\mathcal Y$ has irreducible components $\mathcal Y_1,...,\mathcal Y_r$,
\begin{equation*}
\mathbb P^*_\mathcal{Y}(M|S)=
\cup_{i=1}^r
\mathbb P^*_{\mathcal{Y}_i}(M|S).
\end{equation*}
\end{theorem}

\begin{theorem}\label{T:RCONORMAL-3}
Let $\mathcal L$ be an irreducible germ of a relative Legendrian analytic set of $\mathbb P^*(M|S)$. 
If the analytic set $\pi(\mathcal L)$ is a flat deformation over $S$ of an analytic set of $M$, $\mathcal L=\mathbb P^*_{\pi(\mathcal L)}(M|S)$.
\end{theorem}

 Let $\theta=\xi dx + \eta dy$ be the canonical $1$-form of $T^\ast \C^2 = \C^2 \times \C^2$. 
Hence $\pi=\pi_{\mathbb C^2} : \mathbb{P}^\ast \C^2=\C^2 \times \mathbb{P}^1 \to \C^2$ is given by $\pi(x,y; \xi : \eta)=(x,y)$. 
 Let $U\,[V]$ be the open subset of  $\mathbb{P}^\ast \C^2$ defined by $ \eta \neq 0\,[\xi \neq 0]$. 
 Then $\theta / \eta\,[\theta/\xi]$ defines a contact form $dy-pdx\,[dx-qdy]$ on $U\,[V]$, 
 where $p=- \xi / \eta\,[q=- \eta / \xi]$. 
 Moreover, $dy-pdx$ and $dx-qdy$ define the structure of contact manifold on $\mathbb{P}^\ast \C^2$.

 If $L$ is a germ of a Legendrian curve of $\mathbb P^*M$ and $L$ is not a fiber of $\pi_M$, $\pi_M(L)$ is a germ of plane curve with irreducible tangent cone
 and $L=\mathbb P^*_{\pi_M(L)}M$.

Let $Y$ be the germ of a plane curve with irreducible tangent cone at a point $o$ of a surface $M$.
Let $L$ be the conormal of $Y$. Let $\sigma$ be the only point of $L$ such that $\pi_M(\sigma)=o$.
Let $k$ be the multiplicity of $Y$. Let $f$ be a defining function of $Y$. In this situation we will always choose a system of local coordinates $(x,y)$  of $M$ such that the tangent cone $C(Y)$ of $Y$ equals $\{y=0\}$.

 \begin{lemma}\label{GOODCURVE-3}
The following statements are equivalent$:$
\begin{enumerate}
\item
mult$_\sigma(L)=$mult$_o(Y)$\em ; \em
\item
$C_\sigma(L)\not\supset (D\pi(\sigma))^{-1}(0,0)$\em ; \em
\item
$f\in (x^2,y)^k$\em ; \em
\item
if $t\mapsto (x(t),y(t))$ parametrizes a branch of $Y$, $x^2$ divides $y$.
\end{enumerate}
\end{lemma}

\begin{definition}
Let $S$ be a reduced complex space. Let $Y$ be a reduced plane curve. 
Let $\mathcal Y$ be a deformation of $Y$ over $S$.
We say that $\mathcal Y$ is \em generic \em if its fibers are generic.
If $S$ is a non reduced complex space we say that $\mathcal Y$ is \em generic \em if $\mathcal Y$ admits a generic lifting.
\end{definition}

Given a flat deformation $\mathcal Y$ of a plane curve $Y$ over a complex space $S$ we will denote $\mathbb P_{\mathcal Y}^\ast (\C^2|S)$ by $\mathcal Con(\mathcal Y)$.

\begin{theorem}[Theorem $1.3$, \cite{CN}]\label{L:EQUIEQUI-3}
Let $\chi: (\C^3,0) \to (\C^3,0)$ be a germ of  a contact transformation. Let $L$ be a germ of a Legendrian curve of $\C^3$ at the origin. If $L$ and $\chi(L)$ are in generic position, $\pi(L)$ and $\pi(\chi(L))$ are equisingular.
\end{theorem}

\begin{definition}\label{EQUILEG}
Two Legendrian curves are \emph{equisingular} if their generic plane projections are equisingular. 
\end{definition}

\begin{lemma}
Assume $Y$ is a generic plane curve and $Y\hookrightarrow\mathcal Y$ defines an equisingular deformation of $Y$ with trivial normal cone along its trivial section. Then $\mathcal Y$ is generic.
\end{lemma}

\begin{definition}\label{DEFLEGDEF-3}
Let $L$ be (a germ of) a Legendrian curve of $\mathbb C^3$ in generic position.
 Let $\mathcal L$ be a relative Legendrian curve over (a germ of) a complex space $S$ at $o$. We say that an immersion $i:L\hookrightarrow\mathcal L$ defines a \em deformation 
\begin{equation}\label{LEGMAP-3}
 \mathcal L \hookrightarrow \mathbb C^3\times S \to S
\end{equation}
of the Legendrian curve $L$ over $S$ \em if $i$ induces an isomorphism of $L$ onto $\mathcal L_o$ and there is a generic deformation $\mathcal Y$ of a plane curve $Y$ over $S$  such that
$\chi(\mathcal L)$ is isomorphic to $\mathcal Con \mathcal Y$ by a relative contact transformation verifying (\ref{DIDENTITY3}). 

We say that the deformation (\ref{LEGMAP-3}) is \em equisingular \em if $\mathcal Y$ is equisingular. We denote by $\widehat{\mathcal Def}^{es}_L$ the category of equisingular deformations of $L$.
\end{definition}

\begin{remark}

We do not demand the flatness of the morphism (\ref{LEGMAP-3}). 

\end{remark}

\begin{lemma}
Using the notations of definition $\ref{DEFLEGDEF-3}$, given a section $\sigma:S\to \mathcal L$ of $\mathbb C^3\times S\to S$, there is a relative contact transformation $\chi$ such that $\chi\circ \sigma$ is trivial. Hence $\mathcal L$ is isomorphic to a deformation with trivial section.
\end{lemma}

Consider the maps
$i:X\hookrightarrow X\times S$ and $q: X\times S \to S$.

\begin{theorem}\label{CONTACTEQUI-3}
Assume $\mathcal Y$ defines an equisingular deformation of a generic plane curve $Y$ with trivial normal cone along its trivial section. Let $\chi: X\times S \to X\times S$ be a relative contact transformation verifying 
\begin{equation*}
\chi \circ i=i,  ~~ q \circ \chi=q ~ \hbox{ \rm and } ~ \chi(0,s)= (0,s) ~~
\hbox{ \rm for each } s.
\end{equation*}
 Then $\mathcal Y^\chi=\pi\left(\chi(\mathcal Con \mathcal Y)\right)$ is a generic equisingular deformation of $Y$.
\end{theorem}

\begin{definition}
Let $\mathcal{D}ef^{\,es,\mu}_{f}$ (or $\mathcal{D}ef^{\,es,\mu}_{Y}$)
be the category given in the following way:
the objects of $\mathcal{D}ef^{\,es,\mu}_{f}$
are the objects of $\overset{\twoheadrightarrow}{\mathcal{D}ef^{\,es}_{f}}$;
two objects $\mathcal Y,\mathcal Z$ of $\mathcal{D}ef^{\,es,\mu}_{f}(T)$
are \em isomorphic \em  if
there is a relative contact transformation $\chi$ over $ T $ such that $\mathcal Z=\mathcal Y^\chi$.
\end{definition}

\begin{lemma}\label{CONDUTOR}
Assume $f\in\mathbb C\{x,y\}$ is the defining function of a generic plane curve $Y$.
Let $L$ be the conormal of $Y$.
For each $\ell\ge 1$ there is $h_\ell\in \mathbb C\{x,y\}$ such that
\[
(\ell+1)p^\ell f_x+\ell p^{\ell+1}f_y\equiv h_\ell  \hbox{ mod } I_L.
\]
Moreover, $h_\ell$ is unique modulo $I_Y$.
\end{lemma}

\begin{definition}
Let $f$ be a generic plane curve with tangent cone $\{y=0\}$. 
We will denote by $I_f$ the ideal of $\mathbb C\{x,y\}$ generated by the functions $g$ 
such that $f+\varepsilon g$ is equisingular over $T_\ep$ and has trivial normal cone along its trivial section.
We call $I_f$ the \emph{equisingularity ideal of} $f$.

We will denote by $I^\mu_f$ the ideal of $\mathbb C\{x,y\}$ generated by 
$f,(x,y)f_x$, $(x^2,y)f_y$ and $h_\ell$, $\ell\ge 1$.
\end{definition}

\begin{theorem}[\cite{EDLC}]\label{LAST-3}
Assume $Y$ is a generic plane curve with conormal $L$, defined by a  power series $f$.
Assume $f$ is SQH or $f$ is NND.
If $g_1,...,g_n\in I_f$ represent a basis of $I_f/I^\mu_f$ with Newton order $\geq 1$ , 
the  deformation $\mathcal G$ defined by
\begin{equation}\label{VERSALEQUATION}
G(x,y,s_1,...,s_n)=f(x,y)+\sum_{i=1}^n s_ig_i
\end{equation}
is a semiuniversal deformation of  $f$ in $\mathcal Def^{es,\mu}_f$.
\end{theorem}



\begin{lemma}
Let $S$ be the germ of a complex space. Assume $F$ defines an object $\mathcal{F}$ in  $\overset{\twoheadrightarrow}{\mathcal{D}ef^{\,es}_{f}}(S)$. Given $\gamma \geq 1$ there are $H^\gamma \in \mathcal{O}_S\{x,y\}$ such that
\[
H^\gamma \equiv  p^\gamma \partial _x F \qquad mod \; I_{Con(\mathcal{F})} + \Delta_F.
\]
If $f$ has multiplicity $k$, $H^\gamma \equiv 0$ for $\gamma \geq k-1$.

\end{lemma}
\begin{proof}
Let us first show that
\[
H^\gamma \equiv  (\gamma+1)p^\gamma \partial _x F+ \gamma p^{\gamma + 1} \partial_y F \qquad mod \; I_{Con(\mathcal{F})}.
\]
This is a relative version of Lemma $7.2$ of \cite{EDLC}. Since $\F$ is equisingular, the multiplicity and the conductor are constant. Moreover, there are parametrizations of each component of $\F$. Therefore, we can generalize the argument in the proof of the quoted Lemma. 

Now it is enough to show that  
\begin{equation}\label{SECOND}
\partial_x F+p\partial_y F\equiv 0 \qquad mod \; I_{Con(\mathcal{F})}.
\end{equation}
Assume $\F$ is irreducible. Let $(t,s) \mapsto (X,Y,P)$ be a parametrization of $Con(\F)$. Since $F(X,Y)=0$ we conclude that
\[
\partial_x F \partial_t X + \partial_y F \partial_t Y=0.
\]
Since $P=\partial_t Y/\partial_t X$, (\ref{SECOND}) holds.
\end{proof}

Let $T_\varepsilon$ be the complex space with local ring $\mathbb C\{\varepsilon\}/(\varepsilon^2)$.
Let $I,J$ be ideals of the ring $\mathbb C\{s_1,...,s_m\}$. Assume $J\subset I$. 
Let $X,S,T$ be the germs of complex spaces with local rings $\mathbb C\{x,y,p\}$,
$\mathbb C\{s\}/I,\mathbb C\{s\}/J$. Consider the maps
$i:X\hookrightarrow X\times S$, $j: X\times S \hookrightarrow  X\times T$ and $q: X\times S \to S$.

Let $\frak m_X,\frak m_S$ be the maximal ideals of 
$\mathbb C \{x,y,p\}$, $\mathbb C \{s\}/I$.
Let $\frak n_S$ be the ideal of $\mathcal O_{X\times S}$
generated by $\frak m_X\frak m_S$.

Let $\chi: X\times S \to X\times S$ be a relative contact transformation.
If $\chi$ verifies 
\begin{equation}\label{DIDENTITY3}
\chi \circ i=i,  ~~ q \circ \chi=q ~ \hbox{ \rm and } ~ \chi(0,s)= (0,s) ~~
\hbox{ \rm for each } s.
\end{equation}
 there are 
$\alpha,\beta,\gamma\in \frak n_S$ such that
\begin{equation}\label{abc3}
\chi(x,y,p,s)=(x+\alpha,y+\beta,p+\gamma,s).
\end{equation}

\begin{theorem}\label{T:RCKT3}
$(1)$ Let $\chi : X\times S \to X \times S$ be a relative contact transformation that verifies \em (\ref{DIDENTITY3}). \em
Then $\gamma$ is determined by $\alpha$ and $\beta$. Moreover, there is $\beta_0 \in \mathfrak{n}_S+p\mathcal O_{X\times S}$ such that $\beta$ is the solution of the Cauchy problem
\begin{equation}\label{E:CAUCHY-3}
\left( 1+\frac{\partial \alpha}{\partial x}  + p\frac{\partial \alpha}{\partial y}\right)\frac{\partial \beta}{\partial p} - p\frac{\partial \alpha}{\partial p}\frac{\partial \beta}{\partial y} - \frac{\partial \alpha}{\partial p}\frac{\partial \beta}{\partial x}=p\frac{\partial \alpha}{\partial p},  
\end{equation}
$\beta+p\mathcal O_{X\times S} = \beta_0$.\\
$(2)$ Given $\alpha\in \mathfrak{n}_S$, $\beta_0 \in \mathfrak{n}_S+p\mathcal O_{X\times S}$, 
there is a unique relative contact transformation  $\chi$ that verifies \em (\ref{DIDENTITY3}) \em
and the conditions of statement $(a)$. 
We denote $\chi$ by $\chi_{\alpha,\beta_0}$.\\ 
$(3)$ If  $S=T_\varepsilon$ the Cauchy problem \emph{(\ref{E:CAUCHY-3})}  simplifies into 
\begin{equation}\label{E:CAUCHYTEBETA-3}
\frac{\partial \beta}{\partial p}=p\frac{\partial \alpha}{\partial p}, \qquad \beta +  p\mathcal O_{X\times T_\varepsilon} = \beta_0.
\end{equation}
\end{theorem}

Consider the contact transformations from $\mathbb C^3$ to $\mathbb C^3$ given by
\begin{equation}\label{SEMIS-3}
\Phi(x,y,p)=(\lambda x,\lambda \mu y, \mu p),~~~ \lambda,\mu\in\mathbb C^*,
\end{equation}
\begin{equation}\label{PARABOLOIDAL3}
\Phi(x,y,p)=(ax+bp,y+\frac{ac}{2}x^2+\frac{bd}{2}p^2+bcxp, cx+dp),
~~~\begin{vmatrix}
a & b\\
c & d
\end{vmatrix} 
=1,
\end{equation}

\begin{theorem}\label{ALLCONTACT-3}\em (See \cite{AO} or \cite{Martins}.) \em
Let $\Phi:(\mathbb C^3,0)\to (\mathbb C^3,0)$ the the germ of a contact transformation. Then  $\Phi=\Phi_1\Phi_2\Phi_3$, where $\Phi_1$ is of  type $(\ref{SEMIS-3})$, $\Phi_2$ is of type $(\ref{PARABOLOIDAL3})$ and $\Phi_3$ is of type $(\ref{abc3})$, with 
$\alpha, \beta,\gamma\in\mathbb C\{x,y,p\}$. Moreover, there is $\beta_0\in \mathbb C\{x,y\}$ such that
$\beta$ verifies the Cauchy problem $(\ref{E:CAUCHY-3})$, $\beta-\beta_0\in (p)$   and 
\begin{equation}\label{CCOND-3}
\alpha,\beta,\gamma,\beta_0,\frac{\partial \alpha}{\partial x}, 
\frac{\partial \beta_0}{\partial x}, \frac{\partial \beta}{\partial p},
\frac{\partial^2 \beta}{\partial x\partial p} \in (x,y,p).
\end{equation}
If $D\Phi(0)(\{y=p=0\})=\{ y=p=0 \}$, $\Phi_2=id_{\mathbf C^3}$.
\end{theorem}

\begin{proposition}\label{P:IBS}
Let $f$ and $g$ be two microlocally equivalent SQH or NND generic plane curves. Then, $f$ and $g$ have equisingular semiuniversal microlocal deformations with isomorphic base spaces.
\end{proposition}
\begin{proof}
Let $X,Y$ denote the germs of analytic subsets at the origin of $\C^3$ defined by $Con\, f$ and $Con\, g$ respectively. Let $\chi:\C^3 \to \C^3$ be a contact transformation such that $\chi(Y)=X$ and  $\mathcal X:=(i,\Phi):X \hookrightarrow \C^3\times \C^\ell \to \C^\ell$ be a semiuniversal equisingular deformation of $X$ (to see that such an object exists see Theorem \ref{LAST-3}). Let us show that $(i \circ \chi, \Phi)$ is a semiuniversal equisingular deformation of $Y$:

Let $\mathcal Y:=(j,\Psi):Y \hookrightarrow \C^3\times \C^k \to \C^k$ be an equisingular deformation of $Y$. Because $\mathcal X$ is versal there is $\varphi: \C^k \to \C^\ell$ such that $\varphi^\ast \mathcal X \cong (j \circ \chi^{-1}, \Psi$).
\begin{equation}\label{E:DIAG1-3}
\xymatrix{
Y   \ar@{_{(}->}[d]^{j}  &X \ar[l]_{\chi^{-1}} \ar@{_{(}->}[d]^{\varphi^\ast i}\\
\C^3\times \C^k  \ar[d]^\Psi \ar@{}[r]|*=0[@]{\cong} &\C^3\times \C^k \ar[d]^{\varphi^\ast \Phi} \\
\C^k   \ar@{}[r]|*=0[@]{\cong} &\C^k 
}
\end{equation}
Then, $(\varphi^\ast i \circ \chi,\varphi^\ast \Phi) \cong (j,\Psi)$ which means that $\varphi^\ast (i \circ \chi,\Phi) \cong \mathcal Y$. The result follows from the fact that a semiuniversal deformation is unique up to isomorphism (see Lemma $II.1.12$ of \cite{GLS}).
\end{proof}
Recall that, for a SQH or NND generic plane curve $f$, there is a semiuniversal microlocal equisingular deformation with base space $\C^k$, where $k$ is the the dimension as vector space over $\C$ of $I_f/I^\mu_f$. So, because of Proposition \ref{P:IBS} and Proposition $II.2.17$ of \cite{GLS}, the following defines an invariant between microlocally equivalent fibers of $F$. 

\begin{definition}
Let $f$ be a  SQH or NND generic plane curve. Then 
\[
\widehat{\tau}(f):=dim_\C\, \frac{\C\{x,y\}}{I^\mu_f}
\]
is the \emph{microlocal Tjurina} number of $f$.
\end{definition}



\section{The microlocal Kodaira-Spencer map}\label{MKSM}
Assume $k,n$ are coprime integers, $0<2k<n$. Set $f=y^k-x^n$, $\mu=(n-2)(k-2)$. Consider in $\C[x,y]$ the grading given by $o(x^i y^j)=ki+nj$, $(i,j) \in \mathbb{N}^2$. Set $\omega=o(x^{n-2}y^{k-2})-kn$, $\varpi=o(x^{n-k}y^{k-2})-kn$, $e(x^i y^j)=(i,j) \in \mathbb{N}^2$,
\begin{align*}
B&=\{(i,j) \in \mathbb{N}^2: i\leq n-2,\; j \leq k-2\},\\
C&=\{(i,j) \in B: i+j \leq n-2\},\\
D&=\{(i,j) \in B: o(x^i y^j)-kn \leq \varpi\},\\
A_0&=\{(i,j) \in A: ki+nj > kn \}, \text{for each} \; A \subseteq B.
\end{align*}

Let $m_1,\ldots,m_\mu$ be the family $x^i y^j$, $(i,j) \in B$, ordered by degree. Set $b=\# B_0$. If $\mu-b+1 \leq \ell \leq \mu$, set $o(\ell)=o(m_\ell)-kn$ and
$o(s_{o(\ell)})=-o(\ell)$.

Let $A \subseteq B$. Set $I_A=\{\ell:e(m_\ell) \in A_0\}$, $s_A=\left(s_{o(\ell)}\right)_{\ell \in I_A}$. Set $\C^A=\C^{\#A_0}$ with coordinates $s_A$. Notice that $I_B=\{\mu-b+1,\ldots,\mu\}$.  Moreover,
\[
F_A=f+\sum_{\ell \in I_A} s_{o(\ell)} m_\ell
\]
is homogeneous of degree $kn$.

Let $Y$ be the plane curve defined by $f$. Let $\Gamma$ be the conormal of $Y$. Let $\F_A$ be the deformation defined by $F_A$. Notice that
\begin{itemize}
\item $\F_B$ is a semiuniversal equisingular deformation of $Y$,
\item $\F_C$ is a semiuniversal equisingular microlocal deformation of $Y$,
\item if $C \subseteq A \subseteq B$, $\F_A$ is a complete equisingular microlocal deformation of $Y$.
\end{itemize}

Let $\Delta_{F_A}$ be the ideal of $\C[s_A]$ generated by $\partial_x F_A$ and $\partial_y F_A$. Assume $o(p)=n-k$ in order to guarantee that the contact form $dy-pdx$ is homogeneous.

\begin{lemma}\label{LEMMAH}
Assume $C\subseteq A \subseteq B$ and $\gamma \geq 1$. There is $H^\gamma_A \in \C[s_A]\{x,y\}$ such that $H^\gamma_A \equiv p^\gamma \partial_x F_A \; mod \; I_{Con(\mathcal{F}_A)} + \Delta_{F_A}$ where $H^\gamma_A$ is homogeneous of degree $\gamma(n-k) +kn-k$. If $\gamma \geq k-1$, $H^\gamma_A \in \Delta_{F_A}$. If $C \subseteq A'  \subseteq A  \subseteq B$, $H^\gamma_{A'}=H^\gamma_{A}|_{\C^{A'}}$.

\end{lemma}

\begin{proof}
Set $\psi_0=\theta$, where $\theta^k=-1$. There are $\psi_i \in (s_A)\C[s_A]$, $i \geq 1$, such that 
\[
X(t,s_A)=t^k,\qquad Y(t,s_A)= \sum_{i \geq 0} \psi_i t^{n+i}
\]
defines a parametrization $\Phi$ of $\F_A$. Setting $P(t,s_A)= \sum_{i \geq 0} \frac{n+i}{k}\psi_i t^{n-k+i}$, $X,Y,P$ defines a parametrization $\Psi$ of $Con(\mathcal{F}_A)$. Since $x$ is homogeneous of degree $k$ and $x=t^k$, we assume $t$ homogeneous of degree $1$. Let us show that $Y$ is homogeneous of degree $n$. The $\C^\ast$-action acts on $\Phi$ by
\[
a\cdot \Phi(t,s_A)=\left(a^kt^k, a^n\left( \theta t^n+\sum_{i \geq 1} (a \cdot \psi_i)a^it^{n+i}\right)\right).
\]
Since $F_A$ is homogeneous, for each $s_A$,
\[
t \mapsto \Phi_a(t,s_A)= \left(t^k, \theta t^n+\sum_{i \geq 1} (a \cdot \psi_i)a^it^{n+i}\right)
\]
is another parametrization of the curve defined by $(x,y) \mapsto F_A(x,y,s_A)$. Since the first term of both parametrizations coincide, $\Phi_a=\Phi$, $a \cdot \psi_i=a^{-i}\psi_i$ and $\Phi$ is homogeneous. Therefore, $\Psi$ is homogeneous.

There is an integer $c$ such that $\Phi^\ast(\Delta_{F_A}) \supset t^c \C[s_A]\{t\}$. Remark that $p^\gamma \partial_x F_A$ is homogeneous of degree $\gamma(n-k) +kn-k$. We construct $H^\gamma_A$ in the following manner. There is a monomial $ax^i y^j$, $a \in \C[s_A]$ such that the monomials of lowest $t$-order $\Phi^\ast(ax^i y^j)$ and $\Psi^\ast(p^\gamma \partial_x F_A)$ coincide. Replace $p^\gamma \partial_x F_A$ by $p^\gamma \partial_x F_A-ax^i y^j$ and iterate the procedure. After a finite number of steps we construct $H^\gamma_A$ such that
\[
\Psi^\ast(p^\gamma \partial_x F_A-H^\gamma_A) \in t^c \C[s_A]\{t\}.
\]
Therefore,
\[
p^\gamma \partial_x F_A-H^\gamma_A \in I_{Con(\mathcal{F}_A)} + \Delta_{F_A}.
\]
Remark that the monomial $ax^i y^j$ is homogeneous of degree $\gamma(n-k) +kn-k$.
\end{proof}

Set $\Theta_B= Der_\C \C[s_B]$, $\partial_{o(\ell)}=\partial_{s_{o(\ell)}}$ and $o(\partial_{o(\ell)})=o(\ell)$ for each $\ell \in I_B$. Assume $C \subseteq A' \subseteq A \subseteq B$. Let $\Theta_{A,A'}$ be the $\C[s_A]$-submodule of $\Theta_B$ generated by $\partial_{o(\ell)}$, $\ell \in I_{A'}$. Set $\Theta_A=\Theta_{A,A}$. There are maps
\[
\Theta_A \hookleftarrow \Theta_{A,A'} \overset{r_{A,A'}}{\longrightarrow} \Theta_{A'},
\]
where $r_{A,A'}$ is the restriction to $\C^{A'}$.

\begin{definition}
Let $I^\mu_F$ be the ideal of $\C[s_B][[x,y]]$ generated by $F_B$, $\Delta F_B$ and $H^\gamma_B$, $\gamma=1,\ldots,k-2$. We say that the map
\[
\rho: \Theta_B \to \C[s_B][[x,y]]/I^\mu_F,
\]
given by $\rho(\delta)=\delta F_B + I^\mu_F$ is the \emph{microlocal Kodaira-Spencer map} of $f$. We will denote the kernel of $\rho$ by $\EL_B$.
\end{definition}

Assume we have defined $\EL_A$. We set 
\[
\EL_{A,A'}=\EL_A \cap \Theta_{A,A'} \;\text{and}\; \EL_{A'}=r_{A,A'}(\EL_{A,A'}).
\] 
Let $L$ be a  Lie subalgebra of $\Theta_A$. Consider in $\C^A$ the binary relation $\sim$ given by $p \sim q$ if there is a vector field $\delta$ of $L$ and an integral curve $\gamma$ of $\delta$ such that $p$ and $q$ are in the trajectory of $\gamma$. We denote by $L$ the equivalence relation generated by $\sim$. We say that a subset $M$ of $\C^A$ is an integral manifold of $L$ if $M$ is an equivalence class of $L$.

Assume $C \subseteq A \subseteq B$. The family $m_\ell$, $1\leq \ell \leq \mu$, defines a basis of the $\C[s_A]$-module 
\[
R_A=\C[s_A][[x,y]]/\Delta F_A.
\]
Set $H^0_A=F_A$. The relations
\begin{equation}\label{CIJGA}
m_\ell H^\gamma_A \equiv \sum_{\upsilon=1}^\mu c_{\ell,\upsilon}^\gamma m_\upsilon \qquad mod \; \Delta F_A
\end{equation}
define $c_{\ell,\upsilon}^\gamma \in \C[s_A]$ for each $0 \leq \gamma \leq k-2$, $1 \leq \ell,\upsilon \leq \mu$. Assume $A=B$ and set
\begin{equation}\label{DEFDEL}
\delta_\ell^\gamma = \sum_{\upsilon=\mu-b+1}^\mu c_{\ell,\upsilon}^\gamma \partial_{s_{o(\upsilon)}}, \qquad \ell=1,\ldots,\mu,\; \gamma=0,\ldots, k-2.
\end{equation}
If $m_\ell=x^i y^j$ we will also denote $\delta_\ell^\gamma$ by $\delta_{i,j}^\gamma$. For $1 \leq \gamma \leq k-2$, set
\begin{align*}
&\alpha^0_\ell=o(m_\ell),\quad \alpha^\gamma_\ell= \alpha^0_\ell+\gamma(n-k)-k, \qquad &&\ell=1,\ldots,\mu,\\
&\alpha^0_{i,j}=o(x^i y^j), \quad \alpha^\gamma_{i,j}=\alpha^0_{i,j}+\gamma(n-k)-k, \qquad &&(i,j) \in B.
\end{align*}

\begin{lemma}\label{FIRST}
With the previous notations, we have that:
\begin{enumerate}
\item The vector fields $\delta^\gamma_\ell\; (\delta_{i,j}^\gamma)$ are homogeneous of degree $\alpha^\gamma_\ell\; (\alpha^\gamma_{i,j})$, $0 \leq \gamma \leq k-2$, $1 \leq \ell \leq \mu\; ((i,j) \in B)$.
\item $\delta_{i,j}^\gamma(0)\neq 0$ if and only if $\gamma \geq 1$, $i \leq \gamma-1$, $\gamma + j \leq k-2$.
\item
$\delta_{i,j}^\gamma=0$ if $\alpha^\gamma_{i,j} > \omega$.
\item The Lie algebra $\EL_B$ is generated as $\C[s_B]$-module by $\{ \delta^\gamma_\ell: 0 \leq \gamma \leq k-2, \; \alpha^\gamma_\ell \leq \omega \}$.
\item If $\sigma > \varpi$, $\partial_{s_\sigma} \in \EL_B$.
\item If $(u,v) \in B \setminus C$ there is $\delta \in \EL_B$ such that $\delta=\partial_{s_\sigma}+\ep$ is homogeneous of degree $\sigma=ku+nv-kn$, where $\ep$ is a linear combination of $\partial_{s_{o(i)}}$, $i \in I_B$, $i > \sigma$, with coefficients in $\C[s_B]$.
\end{enumerate}
\end{lemma}

\begin{proof}
(3): Just notice that if $\alpha^0_{i,j} > \omega=o(m_\mu)-kn$ then $o(m_{i,j} F_B) > o(m_\mu)$. Now, because $n >2k$, $o(H^\gamma_B)> kn=o(F_B)$ for any $\gamma=1,\ldots,k-2$, the result holds for $\gamma >0$.

(4): For $\gamma=0$ ($1 \leq \gamma \leq k-2$) and each $\ell =1,\ldots, \mu$ such that $o(m_\ell) \leq \omega$, we have that $\rho(\delta_\ell^\gamma)= \delta_\ell^\gamma F_B + I^\mu_F=m_\ell F_B+I^\mu_F \;(m_\ell H_B^\gamma+I^\mu_F)=0+I^\mu_F$. So, $\{ \delta^\gamma_\ell: 0 \leq \gamma \leq k-2, \; \alpha^\gamma_\ell \leq \omega \} \subset \EL_B$.

Now, let 
\[
\delta=\sum_{\upsilon=\mu-b+1}^{\mu} w_\upsilon \partial_{s_{o(\upsilon)}} \in \Theta_B
\]
such that $\rho(\delta)=0$. Then
\[
\delta F_B=\sum_{\upsilon=\mu-b+1}^{\mu} w_\upsilon m_{\upsilon}= M_0F_B + M_1H_B^1+\ldots+M_{k-2} H_B^{k-2} \; mod \; \Delta F_B,
\]
with $M_0,\ldots,M_{k-2} \in \C[s_B][[x,y]]$. Suppose 
\begin{align*}
M_0 =&\sum_{\ell=1}^{\mu} M_{0,\ell}m_\ell \; mod \; \Delta F_B,\\
&\cdots\\
M_{k-2}=&\sum_{\ell=1}^{\mu} M_{k-2,\ell}m_\ell \; mod \; \Delta F_B,
\end{align*}
where the $M_{\gamma,\ell} \in \C[s_B]$ for each $\ell=1,\ldots,\mu,\, \gamma=0,\ldots,k-2$. Then
\begin{align*}
M_0F_B=&M_{0,1}m_1F_B+\ldots+M_{0,\mu}m_\mu F_B \; mod \; \Delta F_B\\
=&M_{0,1}m_1F_B+\ldots+M_{0,b}m_b F_B \; mod \; \Delta F_B\\
=&M_{0,1}\delta_1^0 F_B + \ldots + M_{0,b} \delta_b^0 F_B \; mod \; \Delta F_B.
\end{align*}
Similarly, for any $\gamma=1,\ldots,k-2$
\begin{align*}
M_\gamma H_B^\gamma=&M_{\gamma,1}m_1H_B^\gamma+\ldots+M_{\gamma,b}m_b H_B^\gamma \; mod \; \Delta F_B\\
=&M_{\gamma,1}\delta_1^\gamma F_B + \ldots + M_{\gamma,b} \delta_b^\gamma F_B \; mod \; \Delta F_B.
\end{align*}
So,
\[
\delta F_B=\sum_{\gamma=0}^{k-2}\sum_{\ell=1}^b M_{\gamma,\ell}\delta_\ell^\gamma F_B \; mod \; \Delta F_B,
\]
which means that
\[
\delta=\sum_{\gamma=0}^{k-2}\sum_{\ell=1}^b M_{\gamma,\ell}\delta_\ell^\gamma.
\]
\end{proof}

Let $L_B$ be the Lie algebra generated by $\delta^\gamma_\ell$, $\gamma=0,\ldots,k-2$, $\ell=1,\ldots,b$. Remark that $\C^B/\EL_B \cong \C^B/L_B$.
Consider a matrix  with lines given by the coefficients of the vector fields $\delta^\gamma_\ell$, $\gamma=0,\ldots,k-2$, $\ell=1,\ldots,b$. After performing Gaussian diagonalization we can assume that:
\begin{itemize}
\item For each $\sigma \in I_B \setminus I_C$ there is a line corresponding to a vector field $\partial_{s_{o(\sigma)}}+\ep$, where $\ep \in \Theta_{B,C}$.
\item The remaining lines correspond to vector fields $\delta'_\ell$, $\ell \in J$, of $\Theta_{B,C}$.
\end{itemize}

The vector fields  $\delta'_\ell$, $\ell \in J$, generate $\EL_{B,C}$ as a $\C[s_B]$-module. Let $\delta_\ell$ be the restriction of $\delta'_\ell$ to $\C^C$ for each $\ell \in J$. The vector fields $ \delta_\ell,\; \ell \in J$,  generate $\EL_C$ as $\C[s_C]$-module. Note that $\{ \delta_\ell,\; \ell \in J\}$ is in general not uniquely determined but the $\C[s_C]$-module generated by them is. Let $L_C$ be the Lie algebra generated by  $\{\delta_\ell,\; \ell \in J\}$. Since $L_C \subseteq L_B$ the inclusion map $\C^C \hookrightarrow \C^B$ defines a map $\C^C/L_C \to \C^B/L_B$. By statement $(6)$ of Lemma \ref{FIRST}, this map is surjective. 

Assume there  is a vector field $\delta_\ell$, $\ell \in J$, of order $\alpha$. Let $\{\delta^{\alpha,i}:i \in I_\alpha\}$ be the set of vector fields $\delta_\ell$, $\ell \in J$, of order $\alpha$, with $I_\alpha=\{1,\ldots, \#I_\alpha\}$. If there is $\ell_0$ such that  $\delta^{\alpha,j}(s_\ell)=0$ for $\ell \leq \ell_0$ and $\delta^{\alpha,i}(s_{\ell_0})\neq0$, we assume that $i<j$. If $I_\alpha=\{1\}$, set $\delta^\alpha=\delta^{\alpha,1}$.

\begin{remark}\label{R:R3.3}
If $k=7, n=15$, we have that a semiuniversal equisingular microlocal deformation of $f$ given by
\begin{align*}
F_C&=y^7+x^{15}+s_2x^{11}y^2+s_3x^9y^3+s_4x^7y^4+s_5x^5y^5+s_{10}x^{10}y^3+s_{11}x^8y^4\\
&+s_{12}x^6y^5+s_{18}x^9y^4+s_{19}x^7y^5+s_{26}x^8y^5.
\end{align*}
Notice that the vector fields $\delta^0_{0,1}$ and $\delta^1_{2,0}$ give origin to the linearly independent vector fields
\[
\delta^{15,1}= 3s_3\partial_{s_{18}}+4s_4\partial_{s_{19}}+\cdots.
\]
and
\[
\delta^{15,2}=\left(\frac{7^2}{15}\left(\frac{4}{7}\psi_2^2-3\left(\frac{15}{7}\right)^2s_4\right)-4s_4\right)\partial_{s_{19}}+\cdots.
\]
\end{remark}

\begin{theorem}\label{BIJECTION}
The map $\C^C/L_C \to \C^B/L_B$ is bijective.
\end{theorem}
\begin{proof}
Let $I_p$ be the subset of $I_B$ that contains $I_C$ and the $p$ smallest elements of $I_B \setminus I_C$. Set $C_p=\{(i,j) \in B : ki+nj-kn \in I_p\}$. The Lie algebra $L_{C_p}=\EL_{C_p} \cup L_B$ generates $\EL_{C_p}$ as $\C[s_{C_p}]$-module. There is $p$ such that $C_p=D$. By statement $(5)$ of Lemma \ref{FIRST} the integral manifolds of $L_B$ are of the type $M \times \C^{B\setminus D}$, where $M$ is an integral manifold of $L_{C_p}$. Therefore, $\C^D/L_D \cong \C^B/L_B$. Assume $\C^{C_{p+1}}/L_{C_{p+1}} \cong \C^B/L_B$ and $I_{C_{p+1}} \setminus {C_{p}}=\{\sigma\}$. The Lie algebra $\EL_{C_{p+1}}$ is generated by $\EL_{C_{p}}$ and a vector field $\partial_{s_{o(\sigma)}}+\ep$, where $\ep \in \EL_{C_{p+1},C}$. Consider the flow of $\partial_{s_{o(\sigma)}}+\ep$  with initial condition at a point of $\C^{C_p}$. We can use this flow to construct an homogeneous affine isomorphism of $\C^{C_{p+1}}$ into itself that equals the identity on $\C^{C_p}$ and rectifies $\partial_{s_{o(\sigma)}}+\ep$, leaving invariant $\EL_{C_p}$. Hence, $\C^{C_{p}}/L_{C_{p}} \cong \C^{C_{p+1}}/L_{C_{p+1}}$.

\end{proof}

\begin{remark}\label{SIMPLIFIED}
 Let  us denote by $P(s_C)$ the restriction of $P \in \C[s_B][[x,y]]$ to $\C^C$. Then, $F_B(s_C)=F_C$, $\Delta F_B(s_C)=\Delta F_C$ and  $H^\gamma_B(s_C)=H^\gamma_C$ for each $\gamma=1, \ldots,k-2$.  Let $\{\delta_{\ell,\mu},\; \ell \in J\} \subset Der_\C\, \C[s_C]$ be the set of vector fields obtained if we proceed as in the definition of $\{\delta'_{\ell},\; \ell \in J\}$, now with $C$ in the place of $B$. Then $<\{\delta_{\ell,\mu}\}>$=$<\{\delta_{\ell}\}>$ as $\C[s_C]$-modules. To see this just notice that, if
\begin{align*}
m_iF_B=&\sum_{j=1}^\mu c^0_{i,j}m_j \; mod \, \Delta F_B\\
m_iH_B^\gamma=&\sum_{j=1}^\mu c^\gamma_{i,j}m_j \; mod \, \Delta F_B
\end{align*}
then
\begin{align*}
m_iF_B(s_C)=&\sum_{j=1}^\mu c^0_{i,j}(s_C)m_j \; mod \, \Delta F(s_C)\\
m_iH_B^\gamma(s_C)=&\sum_{j=1}^\mu c^\gamma_{i,j}(s_C)m_j \; mod \, \Delta F(s_C).
\end{align*}

\end{remark}



\section{Geometric Quotients of Unipotent Group Actions}\label{GQ}

An affine algebraic group is said to be \emph{unipotent} if it is isomorphic to a group of upper triangular matrices of the form $Id+\ep$, where $\ep$ is nilpotent. If $\mathcal G$ is unipotent its Lie algebra $L$ is nilpotent and the map $exp:L \to \mathcal G$ is algebraic. Given a nilpotent Lie algebra $L$, there is a unipotent group $\mathcal G=exp\; L$ such that $L$ is the Lie algebra of $\mathcal G$.

Let $A$ be a Noetherian $\C$-algebra. A linear map $D:A \to A$ is a derivation of $A$ if $D(fg)=fD(g)+gD(f)$. A derivation $D$ of $A$ is nilpotent if for each $f \in A$ there is $n$ such that $D^n(f)=0$. Let $Der^{nil}(A)$ denote the Lie algebra of nilpotent derivations of $A$. Here, we set $A=\C[s_C]$.

Let $\mathcal G$ be an algebraic group acting algebraically on an algebraic variety $X$. If $Y$ is an algebraic variety and $\pi:X \to Y$ a morphism then $\pi$ is called a \emph{geometric quotient}, if
\begin{enumerate}
\item $\pi$ is surjective and open,
\item $(\pi_\ast \mathcal{O}_X)^\mathcal G=\mathcal{O}_Y$,
\item $\pi$ is a orbit map, i.e. the fibres of $\pi$ are orbits of $\mathcal G$.
\end{enumerate}
If a geometric quotient exists it is uniquely determined and we just say that $X/\mathcal G$ exists. Here, $\mathcal G$ will act on each strata of $\C^c=Spec\, A$ through the action of $\mathcal G$ on each fiber of $G$. On Theorem \ref{IMANIFOLD} we prove that $\C^c/\L_c$ is a classifying space for germs of Legendrian curves with generic plane projection $\{y^k+x^n=0\}$. The integral manifolds of ${L_C}$ are the orbits of the action of $\mathcal G_0:=exp\, {L_C}$. Set $L:=[{L_C},{L_C}]$ and $\mathcal G=exp\, L$. Note that $L$ is nilpotent ($\mathcal G$ unipotent) and ${L_C}/L \cong \C\delta_0$, where $\delta_0$ is the Euler field. 

\begin{definition}\label{D:D1.1}
Let $\mathcal G$ be a unipotent algebraic group, $Z=Spec\, A$ an affine $\mathcal G$-variety and $X \subseteq Z$ open and $\mathcal G$-stable. Let $\pi: X \to Y :=Spec \, A^\mathcal G$ be the canonical map. A point $x \in X$ is called \emph{stable} under the action of $\mathcal G$ \emph{with respect to $A$} (\emph{or with respect to $Z$}) if the following holds:\\
There exists an $f \in A^\mathcal G$ such that $x \in X_f=\{y \in X, f(y) \neq 0\}$ and $\pi:X_f \to Y_f:= Spec \, A_f^\mathcal G$ is open and an orbit map. If $X=Z=Spec \,A$ we call a point stable with respect to $A$ just \emph{stable}.
\end{definition}

Let $X^s(A)$ denote the set of stable points of $X$ (under $\mathcal G$ with respect to $A$).

\begin{proposition}[\cite{GP1}]\label{P:P1.1}
With the previous notations, we have that:
\begin{enumerate}
\item $X^s(A)$ is open and $\mathcal G$-stable.
\item $X^s(A)/\mathcal G$ exists and is a quasiaffine algebraic variety.
\item If $V \subset Spec \,A^\mathcal G$ is open, $U=\pi^{-1}(V)$ and $\pi: U \to V$ is a geometric quotient then $U \subset X^s(A)$.
\item If $X$ is reduced then $X^s(A)$ is dense in $X$.
\end{enumerate}
\end{proposition}

\begin{definition}\label{D:D1.2}
A geometric quotient $\pi:X \to Y$ is \emph{locally trivial} if an open covering $\{V_i\}_{i \in I}$ of $Y$ and $n_i \geq 0$ exist, such that $\pi^{-1}(V_i) \cong V_i \times A_{\C}^{n_i}$ over $V_i$.
\end{definition}

We use the following notations:\\
Let $L \subseteq Der^{nil}(A)$ be a nilpotent Lie-algebra and $d: A \to Hom_{\C}(L,A)$ the differential defined by $da(\delta)=\delta(a)$. If $B \subset A$ is a subalgebra then $ \int B := \{a \in A : \delta(a) \in B \;\text{for all} \; \delta \in L\}$. If $\mathfrak{a} \subset A$ is an ideal, $V(\mathfrak{a})$ denotes the closed subscheme $Spec \, A/\mathfrak{a}$ of $Spec\, A$ and $D(\mathfrak{a})$ the open subscheme $Spec \, A - V(\mathfrak{a})$.

Let $A$ be a noetherian $\C$-algebra and $L \subseteq Der^{nil}(A)$  a finite dimensional nilpotent Lie-algebra. Suppose that $A = \cup_{i \in \mathbb{Z}} F^i(A)$ has a filtration
\[
F^\bullet : 0=F^{-1}(A) \subset F^0(A) \subset F^1(A) \subset \ldots
\]
by sub-vector spaces $F^i(A)$ such that 
\begin{flalign*}
(F)& \hspace{2,5cm} \delta F^i(A) \subseteq F^{i-1}(A) \; \text{for all} \; i \in \mathbb{Z}  \; \text{and all} \; \delta \in L. &
\end{flalign*}
Assume, furthermore, that
\[
Z_\bullet : L=Z_0(L) \supseteq Z_1(L) \supseteq \ldots \supseteq Z_\ell(L) \supseteq Z_{\ell+1}(L)=0
\]
is filtered by  sub-Lie-algebras $Z_j(L)$ such that
\begin{flalign*}
(Z)& \hspace{2,5cm} [L,Z_j(L)] \supseteq Z_{j+1}(L)  \; \text{for all} \; j \in \mathbb{Z}. &
\end{flalign*}

The filtration $Z_\bullet$ of $L$ induces projections
\[
\pi_j: Hom_\C(L,A) \to Hom_\C(Z_j(L),A).
\]
For a point $t \in Spec\, A$ with residue field $\kappa(t)$ let
\[
r_i(t):= dim_{\kappa(t)} AdF^i(A) \otimes_A \kappa(t) \qquad i=1, \ldots,\rho,
\]
with $\rho$ minimal such that $AdF^\rho(A)=AdA$,
 \[
s_i(t):= dim_{\kappa(t)} \pi_j(AdA) \otimes_A \kappa(t) \qquad j=1, \ldots,\ell,
\]
such that $s_j(t)$ is the orbit dimension of $Z_j(L)$ at $t$.

Let $Spec \, A =\cup U_\alpha$ be the flattening stratification of the modules 
\[
 Hom_\C(L,A)/AdF^i(A), \qquad  i=1, \ldots,\rho
 \]
and
\[
Hom_\C(Z_j(L),A)/\pi_j(AdA), \qquad  j=1, \ldots,\ell.
 \]

\begin{theorem}[\cite{GP1}]\label{T:T1.2}
Each stratum $U_\alpha$ is invariant by $L$ and admits a locally trivial geometric quotient with respect to the action of $L$. The functions $r_i(t)$ and $s_i(t)$ are constant along $U_\alpha$. Let $x_1,\ldots,x_p \in A, \delta_1,\ldots,\delta_q \in L$ satisfying the following properties:
	\begin{itemize}
	\item there are $\nu_1, \ldots,\nu_\rho, 0 \leq \nu_1 < \ldots < \nu_\rho=p$, such that $dx_1,\ldots,dx_{\nu_i}$ generate the $A$-module $AdF^i(A)$;
	\item there are $\mu_0,\ldots, \mu_\ell, 1=\mu_0<\mu_1<\ldots<\mu_\ell$ such that $\delta_{\mu_j},\ldots,\delta_m \in Z_j(L)$ and $Z_j(L) \subseteq \sum_{i \geq \mu_j} A\delta_i$. 
	\end{itemize}
	Then
	\begin{align}
	rank(\delta_\alpha(x_\beta)(t))_{\beta \leq \nu_i}= &~r_i(t) \qquad i=1, \ldots,\rho, \label{RDR} \\
	rank(\delta_\alpha(x_\beta)(t))_{\alpha \geq \mu_j}= &~s_j(t) \qquad j=1, \ldots,\ell. \label{RDS}
	\end{align}
	The strata $U_\alpha$ are defined set theoretically by fixing (\ref{RDR}) and (\ref{RDS}).
	\end{theorem}

\section{Filtrations and Strata}\label{FS}

Set $L=[L_C,L_C]$. Fix a integer $a$ such that $k\geq a \geq 0$. For each $i \in \mathbb{Z}$ let $F^i_a$ be the $\C$-vector space generated by monomials in $\C[s_C]$ of degree  $ \geq -(a+ik)$. Since $o(\delta) \geq k$ for each homogeneous vector field of $L$, $LF^j_a \subseteq F^{j-1}_a$ for each $j$. For each $m\in \mathbb{Z}$ let $I^m_a$ be the ideal of $\C[[x,y]]$ generated by the monomials of degree $\geq a + mk$. Let $\rho$ be the smallest $i$ such that $dF^i_a$ generates $\C [s_C]d\C [s_C]$ as a $\C [s_C]$-module.

Given $\alpha \in \mathbb{Z}$, set $\alpha^\vee:=nk-k^2-2n-\alpha$. For each integer $j$ set $S_j=\{\alpha : s_{\alpha^\vee} \in F^{\rho-j}_a, \, \alpha \neq 0\}$ and let $Z^a_j$ be the sub-Lie algebra of $L$ generated by the homogeneous vector fields $\delta \in L$ such that $o(\delta)\in S_j$. Remark that
\[
Z^a_1=L,\; Z^a_{\rho+1}=0\qquad \text{and}\qquad [L, Z^a_j] \subseteq Z^a_{j+1}.
\]
For each $t \in \C^C$ let $I^\mu_t$ be the ideal of $\C[[x,y]]$ generated by $F_{t},\Delta F_{t}$ and $H^1_{t},\ldots,H^{k-2}_{t}$. Set
\begin{align*}
\widehat{\tau}_{a,1}^m(t)&=dim_\C\, \C[[x,y]]/(I^\mu_t,I^m_a),\\
\widehat{\tau}_{a,2}^m(t)&=dim_\C\, \C[[x,y]]/(\Delta F_{t},(F_t,H^1_{t},\ldots,H^{k-2}_{t}) \cap I^{\rho-1+2n-m}_a),
\end{align*}
for $m=n,\ldots,n+\rho$ and
\[
\widehat{\tau}_a^\bullet(t)=(\widehat{\tau}_{a,1}^n(t),\ldots, \widehat{\tau}_{a,1}^{n+\rho}(t) ;\widehat{\tau}_{a,2}^n(t), \ldots, \widehat{\tau}_{a,2}^{n+\rho}(t)).
\]
We say that $\widehat{\tau}_a^\bullet(t)$ is the \emph{microlocal Hilbert function} of $X_t$. Set
\begin{align*}
\widehat{\mu}&=\#C=\mu-(k-2)(k-1)/2,\\
 \widehat{\mu}_1^k&=\widehat{\mu}-\#\{m_\ell \in I^k_a: \ell \in I_C\},\\
 \widehat{\mu}_2^k&=\mu-\#\{m_\ell \in I^{\rho-1+2n-k}_a: \ell \in I_{B\setminus C}\}.
 \end{align*}
We only define $\widehat{\tau}_a^\bullet(t)$ for $m=n,\ldots,n+\rho$ because 
\[
\widehat{\tau}_{a,1}^m(t)=\widehat{\tau}_{a,2}^m(t)=\widehat{\tau}(X_t)
\]
(the microlocal Tjurina number of $X_t$) if $m$ is big and
\[
\widehat{\tau}^m_{a,1}(t)= dim_\C\, \C[[x,y]]/I^m_a, \qquad \widehat{\tau}^m_{a,2}(t)= \widehat{\mu}_2^m
\]
(hence independent of $t$) if $m$ is small.

Let $\{U^a_\alpha\}$ be the flattening stratification of $\C^C$ corresponding to $F^\bullet_a$ and $Z^a_\bullet$. It follows from Theorem \ref{T:T1.2}  that $U^a_\alpha \to U^a_\alpha/L$ is a geometric quotient. Moreover, ${L_C}/L \cong \C^\ast$ acts on $U^a_\alpha/L$ and $U^a_\alpha/\mathcal{L}_C= U^a_\alpha/{L_C}$ is a geometric quotient of $U^a_\alpha$ by ${L_C}$. For $t \in \C^C$ let us define 

\[
\underline{e}^a (t)=(u^a_0(t), \ldots, u^a_\rho(t);v^a_0(t), \ldots, v^a_\rho(t)) \in \mathbb{N}^{2\rho+2}, 
\]
where
\[
u^a_j (t)=rank(\delta({s'}_\beta)(t))_{o(r(\beta)) \leq a + jk}, \qquad j=0, \ldots,\rho,
\] 
and
\[
v^a_j (t)=rank(\delta({s'}_\beta)(t))_{ o(\delta) \in S_j}, \qquad j=0, \ldots,\rho.
\]

\begin{lemma}\label{Final1}
The function $t \mapsto \underline{e}^a (t)$ is constant on $U^a_\alpha$ and takes different values for different $\alpha$. The analytic structure of $U^a_\alpha$ is defined by the corresponding subminors of $(\delta({ s_C})(t))$. Moreover, $u^a_j (t)=\widehat{\mu}_1^{n+j}  - \widehat{\tau}_{a,1}^{n+j}(t)$ and $v^a_j (t)=\widehat{\mu}_2^{n+j}-\widehat{\tau}_{a,2}^{n+j}(t)$. In particular, $u^a_\rho(t)=v^a_\rho(t)=\widehat{\mu}-\widehat{\tau}(X_t)$ where $\widehat{\tau}(X_t)$ is the microlocal Tjurina number of the curve singularity $X_t$.
\end{lemma}
\begin{proof}
That $\underline{e}^a (t)$ is constant on $U^a_\alpha$ and takes different values for different $\alpha$ is a consequence of Theorem, \ref{T:T1.2}, as is the claim about the analytic structure of each strata. 

Let $t \in U^a_\alpha$ and consider for each $m \in \{n,\ldots,n+\rho\}$ the induced $\C$-base $\{m_{\ell \in J_m}(t)\}=\{m_{\ell \in J_m}\}$ of $\C\{x,y\}/(\Delta F_t,I^m_a)$. Then, for each $\ell \in J_m$
\[
m_\ell F_t = \sum_{j=1}^b {\delta_{\ell}^0}(s_{o(j)})(t)m_{\mu-b+j} \; mod \, (\Delta F_t,I^m_a)
\]
and
\[
m_\ell H_t^{\gamma} = \sum_{j=1}^b {\delta_{\ell}^\gamma}(s_{o(j)})(t)m_{\mu-b+j} \; mod \, (\Delta F_t,I^m_a)
\]
for $\gamma=1,\ldots,k-2$. Then, by definition of $\widehat{\tau}_a^\bullet(t)$ and from the definition of $\{\delta_\mu\}$, $u^a_j (t)= \widehat{\mu}_1^{n+j}  - \widehat{\tau}_{a,1}^{n+j}(t)$.

The proof of the claim about the $v^a_j (t)$ is similar with the difference that we're now interested in the relations $mod\; \Delta F_t$ between the $m_\ell F_t, m_\ell H_t^{\gamma}$ that belong to $I^{\rho-1+2n-m}_a$ for each $m \in \{n,\ldots,n+\rho\}$. Note that $m_\ell F_t, m_\ell H_t^{\gamma} \in I^{\rho-1+2n-m}_a$ if and only if $\alpha^0_\ell, \alpha^\gamma_\ell \in S_{m-n}$.

\end{proof}

\begin{lemma}\label{IMANIFOLFPREV}
If $a,b \in \C^B$ are such that $\mathcal{C}on(\F_a) \cong \mathcal{C}on(\F_b)$, there is $\psi:\C \to \C^B$ microlocally trivial such that $\psi(0)=a$ and $\psi(1)=b$.
\end{lemma}
\begin{proof}
Let $\chi_0$ be a contact transformation given by $\alpha,\beta_0$ such that $F_b=uF^{\chi_0}_a$ for some unit $u \in \C\{x,y\}$. We can assume $deg\; \chi_0 > 0$. There is a relative contact transformation $\chi(t)$ over $\C$ such that $\chi(0)=id_{\C^3}$ and $\chi(1)=\chi_0$. Then
\[
G(t)=u(tx,ty)F_B^\chi(x,y,a)
\]
is an unfolding of $F_a$ such that $G(1)=F_b$. By versality of $F_B$ and because $F_a$ is semiquasihomogeneous ($j(F_a)=(F_a,j(F_a))$) there is a relative coordinate transformation 
\begin{align*}
\Phi: &\C\times \C^2 \to \C \times \C^2\\
&(t,x,y) \mapsto (t,\Phi_1,\Phi_2)
\end{align*}
and $\psi:\C \to \C^B$ such that
\[
\Phi\left(G(t)\right)=F_{\psi(t)}.
\]
(see Remark $1.1$ and Corollary $3.3$ of \cite{MSSS}). Now, because $F_B$ is semiuniversal (hence does not contain trivial subfamilies with respect to right equivalence) $\Phi(1)\left(G(1)\right)=\Phi(1)\left(F_b\right)=F_{\psi(1)}$ implies that $\psi(1)=b$. 

\end{proof}

\begin{theorem}\label{IMANIFOLD}
Given $a,b \in \C^C$, $\mathcal{C}on(\F_a) \cong \mathcal{C}on(\F_b)$ if and only if $a$ and $b$ are in the same integral manifold of $\EL_C$.
\end{theorem}
\begin{proof}
By Theorem \ref{BIJECTION} we can replace $C$ by $B$.

Let us first prove sufficiency.  Let $C \subset A \subset B$ and $S$ be a complex space. We say that a holomorphic map $\psi: S \to \C^A$ is \emph{trivial} if for each $o \in S$, $\psi^\ast \F_A$ is a trivial deformation of $\mathcal{D}ef^{\,es,\mu}_{f} (S,o)$. Assume $\psi:(\C,0) \to \C^B$ is the germ of an integral curve of a vector field $\delta$ in $\EL_B$. Set $q=\psi(0)$. Let $\psi_\ep: T_\ep \to \C^B$ be the morphism induced by $\psi$. There are $a_0, a_1, \ldots,a_l,\alpha_0,\beta_0 \in \C\{s_B\}[[x,y]]$ such that
\[
\delta F_B=a_0F_B+\sum_{j=1}^\ell a_j H_B^j+\alpha_0 \partial_x F_B+\beta_0 \partial_y F_B.
\]
Set $u=1+\ep a_0(q)$, $\alpha=\alpha(q) + \sum_{j=1}^\ell a_j(q)p^j$ and $\beta=\beta(q) + \sum_{j=1}^\ell \frac{j}{j+1} a_j(q)p^{j+1}$.
By Theorem \ref{T:RCKT3} there is $\gamma \in \C\{x,y,p\}$ such that
\[
(x,y,p,\ep) \mapsto (x+\alpha\ep,y+\beta\ep,p+\gamma\ep,\ep)
\]
defines a relative contact transformation $\chi^\ep$ over $T_\ep$. Let $G \in \C\{x,y,p,\ep\}$ be defined by $G(x,y,p,\ep)=F_B(x+\alpha\ep,y+\beta\ep,q)$. Since $\psi^\ast \F_B\equiv uG \; mod\; (\ep)$ and 
\[
\partial_\ep \psi^\ast \F_B\equiv \partial_\ep uG \;mod \; I_{\mathcal{C}on(\F_q)}+(\ep),
\]
we have that
\[
\psi^\ast \F_B\equiv uG \; mod\; I_{\chi^\ep(\mathcal{C}on(\F_q))}+(\ep^2).
\]
Therefore, $\psi_\ep^\ast F_B$ is a trivial deformation of $\mathcal{D}ef^{\,es,\mu}_{f} (T_\ep)$. Then $\Psi^\ast F_B$ is a trivial deformation of $\mathcal{D}ef^{\,es,\mu}_{f} (\C,0)$ (see the proof of Theorem \ref{LAST-3}).

Conversely, assume that there is a germ of contact transformation $\chi_1$ such that $(F_a^{\chi_1})=(F_b)$. We can assume $deg\; \chi_1 > 0$. If $\chi_1$ is of type (\ref{SEMIS-3}), by Lemma \ref{IMANIFOLFPREV} there is a trivial curve $\psi:\C \to \C^B$ such that $\psi(0)=a$ and $\psi(1)=b$. Moreover, $\psi$ is an integral curve of the Euler vector field. Since the derivative of $\chi_1$ leaves $\{y=p=0\}$ invariant, we can assume by Theorem \ref{ALLCONTACT-3} that $\chi_1$ is of type (\ref{abc3}). Set $\chi=\chi_{t\alpha,t\beta_0}$. There is a curve with polynomial coefficients $\psi:\C \to \C^B$ such that $F_a^\chi=\psi^\ast F_B$, $\psi(0)=a$ and $\psi(1)=b$.

Let $\Omega$ be an open set of $\C$. Let $\psi:\Omega \to \C^B$ be a trivial curve. Let us show that $\psi$ is contained in an integral manifold of $\EL_B$. Let $U$ be the union of the strata $U_\alpha$ such that, for each $c \in U$ the microlocal Tjurina number of $F_c$ equals the microlocal Tjurina number of $F_a$. Remark that the trajectory of $\psi$ is contained in $U$. By Theorem \ref{T:T1.2} $\EL_{B|U}$ verifies the Frobenius Theorem. Hence, it is enough to show that, for each $t_0 \in \Omega$, there is $\delta \in \EL_B$ such that $\psi'(t_0)=\delta(\psi(t_0))$. We can assume $t_0=0$. Since $\psi$ is trivial, there are a relative contact transformation $\chi$ and $u \in \C\{x,y,t\}$ such that $u(x,y,0)=1$ and
\[
F(x,y,\psi(t))\equiv uF^\chi(x,y,q) \; mod \; I_{\chi\left(\mathcal{C}on(\F_{q})\right)}.
\]
If $\chi$ is of type \ref{SEMIS-3}, we can assume $\delta$ is the Euler field.
Hence we can assume that $\chi$ is of type (\ref{abc3}). Therefore there are $\ell \geq 1$ and $a,b,a_i \in \C\{x,y\}$, $1\leq i \leq \ell$, such that
\[
 F(x,y,\psi(t))=uF(x,y,q)+\sum_{\ell=1}^{k-2} a_\ell tH_{q}^\ell +at\partial_x F_q +bt\partial_y F_q    \;mod \:(t^2).
\]
Deriving in order to $t$ and evaluating at $0$, there is $a_0 \in \C\{x,y\}$ such that
\[
\textstyle \sum_{(i,j) \in C_0} \psi'_{i,j}(0)x^iy^j=a_0F_q+\sum_{\ell=1}^{k-2} a_\ell H_{q}^\ell +a\partial_x F_q +b\partial_y F_q.
\]
There are $\delta \in L_B$ and $\ep \in \Delta_{F_B}$ such that
\[
\textstyle \delta F_B=a_0F_B + \sum_{\ell=1}^{k-2} a_\ell H_{B}^\ell + \ep.
\]
Hence 
\[
\textstyle \sum_{(i,j) \in B_0} \psi'_{i,j}(0)x^iy^j -  \delta(q) F_B=\ep(q)+a\partial_x F_q +b\partial_y F_q.
\]
If $\delta=\sum_{(i,j) \in B_0} a_{i,j}\partial_{s_{i,j}}$, $a_{i,j}(\psi(0))=\psi'_{i,j}(0)$ for each $(i,j) \in B_0$.
\end{proof}

\begin{theorem}\label{T:T2.1}
$(1)$ Let $\underline{e}=(e_1, \ldots,e_\rho) \in \mathbb{N}^{\rho+1}$ and let $U^a_{\underline{e}}$ denote the unique stratum (assumed to be not empty) such that $\underline{e}^a(t)=\underline{e}$ for each $t \in U^a_{\underline{e}}$. The geometric quotient $U^a_{\underline{e}}/\mathcal{L}$ is quasiaffine and of finite type over $\C$. It is a coarse moduli space for the functor which associates to any complex space $S$ the set of isomorphism classes of flat families (with section) over $S$ of plane curve singularities with fixed semigroup  $\langle k,n \rangle$ and fixed microlocal Hilbert function $\widehat{\tau}_a^\bullet$.\\
$(2)$ Let $T_{\widehat{\tau}_{min}}$ be the open dense set defined by singularities with minimal microlocal Tjurina number $\widehat{\tau}_{min}$. Then the geometric quotient $T_{\widehat{\tau}_{min}}/\mathcal{L}_C$ exists and is a coarse moduli space for curves with semigroup $\langle k,n \rangle$ and microlocal Tjurina number $\widehat{\tau}_{min}$. Moreover, $T_{\widehat{\tau}_{min}}/\mathcal{L}_C$ is locally isomorphic to an open subset of a weighted projective space.
\end{theorem}

\begin{proof}
It follows from Lemma \ref{Final1} and Theorems \ref{T:T1.2} and \ref{IMANIFOLD}. 
\end{proof}

\section{Example}

The function 
\begin{equation*}
F_C=y^6+x^{13}+s_2x^{9}y^2+s_3x^7y^3+s_4x^5y^4+s_9x^8y^3+s_{10}x^{6}y^4+s_{16}x^7y^4.
\end{equation*}
is a semiuniversal equisingular microlocal deformation of $f=y^6+x^{13}$.

The Lie algebra $L_C$ is generated by the vector fields
\begin{align*}
\delta^0 &= 2s_2\partial_{s_2}+3s_3\partial_{s_3}+4s_4\partial_{s_4}+9s_9\partial_{s_9}+10s_{10}\partial_{s_{10}}+16s_{16}\partial_{s_{16}}\\
\delta^6 &=3s_3\partial_{s_9}+(4s_4-\frac{58}{39}s_2^2)\partial_{s_{10}}+10s_{10}\partial_{s_{16}}\\
\delta^7 &=2s_2\partial_{s_9}+3s_3\partial_{s_{10}}\\
\delta^{12} &=4s_{4}\partial_{s_{16}}\\
\delta^{13} &=3s_{3}\partial_{s_{16}}\\
\delta^{14} &=2s_{2}\partial_{s_{16}}.
\end{align*}
Choosing $a=6$ we get $F^0_a=\langle s_2,s_3,s_4\rangle$, $F^1_a=\langle s_2,s_3,s_4,s_9,s_{10}\rangle$, $F^2_a=\langle s_2,s_3,s_4,s_9,s_{10},s_{16}\rangle$. So, $\rho=2$ and the stratification $\{U^a_\alpha\}$ given by fixing $\underline{e}^a (t)=(u^a_0(t), u^a_1(t), u^a_2(t);v^a_0(t), v^a_1(t), v^a_2(t))$ is given by
\begin{align*}
U_1=&\{t=(t_2,t_3,t_4,t_9,t_{10},t_{16}) \in Spec\, \C [{ s_C}] : \underline{e}^a (t)=(1,3,4;1,3,4)\}\\
=&\{t  : 9t_3^2-8t_2t_4+\frac{116}{39}t_2^3\neq 0\}.\\
U_2=&\{t \in Spec\, \C [{s_C}] : \underline{e}^a (t)=(1,2,3;1,2,3)\}\\
=&\{t : 9t_3^2-8t_2t_4+\frac{116}{39}t_2^3= 0 \;\text{and}\; t_2 \neq 0 \;\text{or}\; t_3 \neq 0 \;\text{or}\; t_4 \neq 0\}.\\
U_3=&\{t \in Spec\, \C [{s_C}] : \underline{e}^a (t)=(0,1,2;0,1,2)\}\\
=&\{t : t_2=t_3=t_4=0  \;\text{and}\; t_{10} \neq 0\}.\\
U_4=&\{t \in Spec\, \C [{s_C}] : \underline{e}^a (t)=(0,1,1;0,0,1)\}\\
=&\{t : t_2=t_3=t_4=t_{10}=0  \;\text{and}\; t_{9} \neq 0\}.\\
U_5=&\{t \in Spec\, \C [{s_C}] : \underline{e}^a (t)=(0,0,1;0,0,1)\}\\
=&\{t : t_2=\cdots=t_{10}=0  \;\text{and}\; t_{16} \neq 0\}.\\
U_6=&\{t \in Spec\, \C [{s_C}] : \underline{e}^a (t)=(0,0,0;0,0,0)\}\\
=&\{t : t_2=\cdots=t_{16}=0\}.
\end{align*}

$U_1$ is the stratum with minimal microlocal Tjurina number.


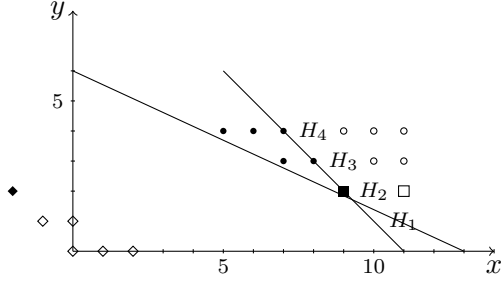
\begin{figure}
\captionsetup{singlelinecheck=off}
\begin{tikzpicture}[scale=0.4,inner sep =0pt]
\draw[->] (0,0) -- (14,0);
\draw[->] (0,0) -- (0,8);
\draw (14,-0.5) node {$x$};
\draw (-0.5,8) node {$y$};
\draw (11,0) -- (5,6);
\draw (0,6) -- (13,0);
\foreach \x in {3,...,14} \draw (\x cm,2pt) -- (\x cm,-2pt) node[anchor=north]{};
\draw (10,-0.5) node {\scriptsize $10$};
\draw (5,-0.5) node {\scriptsize $5$};
\foreach \y in {2,...,6} \draw (2pt,\y cm) -- (-2pt,\y cm) node[anchor=north] {};
\draw (-0.5,5) node {\scriptsize $5$};

\foreach \i in {5,...,7}
{
\path (\i,4) coordinate (A\i);
\fill (A\i) circle (3pt);
}
\foreach \i in {7,...,8}
{
\path (\i,3) coordinate (B\i);
\fill (B\i) circle (3pt);
}

\foreach \i in {9,...,9}
{
\node[minimum size=4pt] (C\i) at (\i,2) [fill] {};
}

\draw (8,4) node {\scriptsize $H_4$};
\draw (9,3) node {\scriptsize $H_3$};
\draw (10,2) node {\scriptsize $H_2$};
\draw (11,1) node {\scriptsize $H_1$};

\foreach \l in {9,...,11}
{
\path (\l,4) coordinate (A\l);
\draw (A\l) circle (3pt);
}
\foreach \l in {10,...,11}
{
\path (\l,3) coordinate (B\l);
\draw (B\l) circle (3pt);
}

\foreach \i in {11,...,11}
{
\node[minimum size=4pt] (C\i) at (\i,2) [draw]  {};
}

\foreach \i in {0,...,2}
{
\node[diamond,minimum size=4pt] (E\i) at (\i,0) [draw] {};
}
\foreach \i in {-1,...,0}
{
\node[diamond,minimum size=4pt] (D\i) at (\i,1) [draw] {};
}

\foreach \i in {-2,...,-2}
{
\node[diamond,minimum size=4pt] (C\i) at (\i,2) [fill] {};
}
\end{tikzpicture}
\caption{This figure concerns the case $k=6$ and $n=13$. The diamonds represent the set of orders of the vector fields generating ${L_C}$. The black circles and black squares represent $C_0$. The leading monomials of $H_1,\ldots,H_4$ are represented as well. The white square represents the leading monomial of $xH_2$ and $yH_1$ which produce the vector field $\delta^{14}$ with order represented by a black diamond. The order of the vector fields $\delta^0_{0}$, $\delta^0_{1}$, $\delta^1_{1}$, $\delta^0_{2}$ and $\delta^0_{3}$ are represented by white diamonds.
} \label{fig:M2}
\end{figure}

Let us present detailed calculations concerning the generators of $L_C$ in the previous example.
Let $Y$ denote the germ at the origin of $\{F_C=0\}$. The relative conormal $\EL$ of $Y$ can be parametrized by
\begin{align*}
x&=-t^6,\\
y&=t^{13}+\psi_2t^{15}+\psi_3t^{16}+\psi_4t^{17}+\psi_5t^{18}+\psi_6t^{19}+\psi_7t^{20}+\psi_8t^{21}+\psi_9t^{22}+\cdots,\\
p&=-\frac{13}{6}t^7-\frac{5}{2}\psi_2t^{9}-\frac{8}{3}\psi_3t^{10}-\frac{17}{6}\psi_4t^{11}-3\psi_5t^{12}-\frac{19}{6}\psi_6t^{13}-\frac{10}{3}\psi_7t^{14}-\frac{7}{2}\psi_8t^{15}\\
&-\frac{11}{3}\psi_9t^{16}+\cdots,
\end{align*}
where $\psi_i \in ({s_C})\C[{s_C}]$ are homogeneous of degree $-i$. These are the $a_i$ such that the polynomial in $\C[t]$ given by the following SINGULAR session is zero:\\ 
\phantom{}\\
{\tt
> ring r=(0,a2,a3,a4,a5,a6,a7,a8,a9,s2,s3,s4,s9,s10,s16),(x,y,t),dp;\\
> poly F=y6+x13+s2*x9y2+s3*x7y3+s4*x5y4+s9*x8y3+s10*x6y4+s16*x7y4;\\
> subst(F,x,-t6);\\
-t}\^{\tt 78} {\tt+(-s2)*y}\^{\tt 2}{\tt *t}\^{\tt 54}{\tt+(s9)*y}\^{\tt 3*t}\^{\tt 48}{\tt +(-s16)*y}\^{\tt 4*t}\^{\tt 42}{\tt +(-s3)*y}\^{\tt 3*t}\^{\tt 42}\\
{\tt +(s10)*y}\^{\tt 4*t}\^{\tt 36+(-s4)*y}\^{\tt 4*t}\^{\tt 30+y}\^{\tt 6}\\
{\tt
> subst(-t}\^{\tt 78} {\tt+(-s2)*y}\^{\tt 2}{\tt *t}\^{\tt 54}{\tt+(s9)*y}\^{\tt 3*t}\^{\tt 48}{\tt +(-s16)*y}\^{\tt 4*t}\^{\tt 42}{\tt +(-s3)*y}\^{\tt 3*t}\^{\tt 42}\\
{\tt +(s10)*y}\^{\tt 4*t}\^{\tt 36+(-s4)*y}\^{\tt 4*t}\^{\tt 30+y}\^{\tt 6,y,t}\^{\tt 13+a2*t}\^{\tt 15+a3*t}\^{\tt 16+a4*t}\^{\tt 17+a5*t}\^{\tt 18}\\
{\tt +a6*t}\^{\tt 19+a7*t}\^{\tt 20+a8*t}\^{\tt 21+a9*t}\^{\tt 22)\\
}\\
\phantom{}\\
As we'll see, the only $\psi_i$ we actually need to find the generators of ${L_C}$ is 
\[
\psi_2=s_2/6.
\]
Let us calculate the vector fields generating ${L_C}$. Here, all equalities are $mod\, \Delta F_C$ and in the vector fields we identify, by abuse of language, the monomials and the corresponding $\partial$'s :
\begin{itemize}
\item $\delta^{14}$:\\
\[
xH_2=xp^2\partial_x F_C=13p^2x^{13}+9s_2p^2x^9y^2+\cdots
\]
Notice that, as a consequence of Lemma \ref{FIRST},  the monomials occurring with order bigger than  $deg(x^7y^4)$ can be ignored in this calculation. From now on, whenever we use the symbol $\cdots$ we mean that bigger order monomials can be ignored.
Now, continuing the previous SINGULAR session:\\
\phantom{}\\
{\tt
> poly p=(-13t7-15*a2*t9-16*a3*t10-17*a4*t11-18*a5*t12-19*a6*t13-20*a7*t14\\
-21*a8*t15-22*a9*t16)/6;\\
> poly X=-t6;\\
> poly Y=t13+a2*t15+a3*t16+a4*t17+a5*t18+a6*t19+a7*t20+a8*t21+a9*t22;\\
> p}\^{\tt 2*X}\^{\tt13-(13/6)}\^{\tt 2*X}\^{\tt 11*Y}\^{\tt 2;\\
(-35*a9}\^{\tt 2)/4*t}\^{\tt 110+(-293*a8*a9)/18*t}\^{\tt 109+(-271*a7*a9-136*a8}\^{\tt 2)/18*t}\^{\tt 108+(-249*a6*a9-251*a7*a8)/18*t}\^{\tt 107+(-454*a5*a9-460*a6*a8-231*a7}\^{\tt 2)/36*t}\^{\tt 106+(-205*a4*a9-209*a5*a8-211*a6*a7)/18*t}\^{\tt 105+(-183*a3*a9-188*a4*a8-191*a5*a7-96*a6}\^{\tt 2)/18*t}\^{\tt 104+(-161*a2*a9-167*a3*a8-171*a4*a7-173*a5*a6)/18*t}\^{\tt 103+(-292*a2*a8-302*a3*a7-308*a4*a6-155*a5}\^{\tt 2)/36*t}\^{\tt 102+(-131*a2*a7-135*a3*a6-137*a4*a5-117*a9)/18*t}\^{\tt 101+(-116*a2*a6-119*a3*a5-60*a4}\^{\tt 2-104*a8)/18*t}\^{\tt 100 +(-101*a2*a5-103*a3*a4-91*a7)/18*t}\^{\tt 99+(-172*a2*a4-87*a3}\^{\tt 2-156*a6)/36*t}\^{\tt 98\\+(-71*a2*a3-65*a5)/18*t}\^{\tt 97+(-14*a2}\^{\tt 2-26*a4)/9*t}\^{\tt 96+(-13*a3)/6*t}\^{\tt 95+(-13*a2)/9*t}\^{\tt 94
}\\
\phantom{}\\
we see that
\[
p^2x^{13}=\left(\frac{13}{6}\right)^2x^{11}y^{2}+\frac{13\psi_2}{9}x^7y^4+\cdots
\]
Now, $\delta^1_3$, given by $yH_1$, which has the same order as $xH_2$ can be used to, through elementary operations, eliminate from  $\delta^2_1$ the monomial $x^{11}y^{2}$. Thus, 
\[
\delta^{14}=s_2x^7y^4.
\]
\item  $\delta^{13}$:\\
\[
-6.13yF_C=2s_2x^9y^3+3s_3x^7y^4+\cdots
\]
But, as
\[
H_3=\left(\frac{13}{6}\right)^3.13x^9y^3+\cdots
\]
we see that, through elementary operations involving $\delta^3_0$, we can eliminate from  $\delta^0_3$ the monomial $x^{9}y^{3}$.
Thus, 
\[
\delta^{13}=3s_3x^7y^4.
\]
\item $\delta^{12}$:\\
\[
-6.13x^2F_C=2s_2x^{11}y^2+3s_3x^9y^3+4s_4x^7y^4+\cdots
\]
through elementary operations involving $\delta^1_3$ and $\delta^3_0$ we can eliminate the monomials $x^{11}y^2$ and $x^9y^3$ from $\delta^0_2$ and get:
\[
\delta^0_2=(4s_4+\ast s_2^2)x^7y^4, \qquad \ast \in \C.
\]
Finally, using $\delta^{14}$ to eliminate $\ast s_2^2x^7y^4$, we have that
\[
\delta^{12}=4s_4x^7y^4.
\]
\item $\delta^{7}$:\\
\[
xH_1=xp\partial_x F_C=13px^{13}+9s_2px^{9}y^2+7s_3px^7y^3+5s_4px^5y^4+8s_9px^8y^3+\cdots
\]
and
\[
px^{13}=\frac{13}{6}x^{12}y+\frac{s_2}{18}x^8y^3+\frac{s_3}{12}x^6y^4+\cdots
\]
\begin{remark}\label{ignore}
The reason why we can ignore in $px^{13}$ the monomials that occur after $x^6y^4$ is that
\begin{enumerate}
\item All monomials after $x^6y^4$, except for $x^7y^4$,  can be eliminated because of Lemma \ref{FIRST} and through elementary operations involving $\delta^1_3$ and $\delta^3_0$.
\item Even $x^7y^4$ can be ignored, observing that $px^{13}$ is homogeneous of degree $7$ and as such, the only variables involved in the coefficient (in $\C[s_C]$) of $x^7y^4$ may be $s_2$, $s_{3}$ or $s_4$. Now, using $\delta^{14}$, $\delta^{13}$ and $\delta^{12}$ we can eliminate, through elementary operations, the monomial $x^7y^4$ from $\delta^{7}$.
\end{enumerate}
\end{remark}
From
\[
y\partial_x F_C=13x^{12}y+9s_2x^8y^3+7s_3x^6y^4+5s_4x^4y^5+8s_9x^7y^4+\cdots
\]
we get that
\[
\frac{13}{6}x^{12}y=-\frac{3}{2}s_2x^8y^3-\frac{7}{6}s_3x^6y^4-\frac{5}{6}s_4x^4y^5-\frac{8}{6}s_9x^7y^4+\cdots
\]
Reasoning as in remark \ref{ignore} we see that $s_4x^4y^5$ can be ignored. Thus,
\[
13px^{13}=13\left(\left(-\frac{3}{2}s_2+\frac{s_2}{18}\right)x^8y^3+\left(-\frac{7}{6}s_3+\frac{s_3}{12}\right)x^6y^4-\frac{8}{6}s_9x^7y^4+..\right)
\]
Now,
\[
px^9y^2=\frac{13}{6}x^8y^3+\cdots
\]
\[
px^7y^3=\frac{13}{6}x^6y^4+...
\]
\[
px^8y^3=\frac{13}{6}x^7y^4+...
\]
Once again, the monomials ignored can be eliminated, reasoning as in Remark \ref{ignore}. So,
\begin{align*}
xH_1=&\left(13\left(-\frac{3}{2}s_2+\frac{s_2}{18}\right)+9\frac{13}{6}s_2\right)x^8y^3+\left(13\left(-\frac{7}{6}s_3+\frac{s_3}{12}\right)+7\frac{13}{6}s_3\right)x^6y^4+\\
&+\left(-13\frac{8}{6}s_9+8\frac{13}{6}s_9\right)x^7y^4\\
=&\frac{13}{18}s_2x^8y^3+\frac{13}{12}s_3x^6y^4.
\end{align*}
We get that
\[
\delta^{7}=\frac{s_2}{3}x^8y^3+\frac{s_3}{2}x^6y^4.
\]
\item $\delta^{6}$:\\
\[
-6.13xF_C=2s_2x^{10}y^2+3s_3x^8y^3+4s_4x^6y^4+9s_9x^9y^3+10s_{10}x^7y^4+\cdots
\]
Because (monomials ignored as in Remark \ref{ignore})
\[
H_2=p^2\partial_x F_C=13p^2x^{12}+9s_2p^2x^8y^2+\cdots,
\]
\[
p^2x^{12}=\left(\frac{13}{6}\right)^2x^{10}y^2+\frac{13}{6.9}s_2x^6y^4+\cdots,
\]
and
\[
p^2x^8y^2=\left(\frac{13}{6}\right)^2x^{6}y^4+\cdots
\]
wet get
\begin{align*}
&-6.13xF_C-2s_2\left(\frac{6}{13}\right)^2\frac{H_2}{13}=\\
=&3s_3x^8y^3+\left(4s_4-2s_2\left(\frac{6}{13}\right)^2\frac{13}{6.9}s_2-2s_2\left(\frac{6}{13}\right)^2\frac{9}{13}\left(\frac{13}{6}\right)^2s_2\right)x^6y^4+10s_{10}x^7y^4=\\
=&3s_3x^8y^3+\left(4s_4-\frac{58}{39}s_2^2\right)x^6y^4+10s_{10}x^7y^4.
\end{align*}
So,
\[
 \delta^{6}=3s_3x^8y^3+\left(4s_4-\frac{58}{39}s_2^2\right)x^6y^4+10s_{10}x^7y^4.
\]
\end{itemize}

\end{document}